\documentclass[12pt, reqno, a4paper]{amsart}
\usepackage{graphicx, epsfig, psfrag}
\usepackage{amsfonts,amsmath,amssymb,amsbsy,amsthm}
\usepackage{bm}
\usepackage{xcolor}
\usepackage{float}
\usepackage{mathrsfs}

\usepackage{environments}

\renewcommand{\cases}[1]{\left\{ \begin{array}{rl} #1 \end{array} \right.}
\newcommand{\smfrac}[2]{{\textstyle \frac{#1}{#2}}}

\def\R{\mathbb{R}}

\def\Z{\mathbb{Z}}

\def\CC{{\rm C}}

\def\<{\langle}
\def\>{\rangle}

\def\eps{\varepsilon}
\def\argmin{{\rm argmin}}

\def\conv{{\rm conv}}

\def\h{{\rm qc}}
\def\cb{{\rm cb}}
\def\tot{{\rm tot}}

\def\As{\mathcal{A}}
\def\Cs{\mathcal{C}}
\def\Asper{\mathcal{A}_{\#}}

\def\Is{\mathcal{I}}

\def\Fs{\mathscr{F}}

\def\Ys{\mathscr{Y}}
\def\Us{\mathscr{U}}

\begin{document}

\title[A Priori and A Posteriori Analysis of the QNL-QC in 1D]{A
  Priori and A Posteriori Analysis of the Quasi-Nonlocal
  Quasicontinuum Method in 1D}

\author{C. Ortner}
\address{C. Ortner\\ Mathematical Institute\\
  24-29 St Giles' \\ Oxford OX1 3LB \\ UK}
\email{ortner@maths.ox.ac.uk}

\date{\today}

\thanks{This work was supported by the EPSRC Critical Mass Programme
  ``New Frontiers in the Mathematics of Solids''.}

\subjclass[2000]{65N12, 65N15, 70C20}

\keywords{quasicontinuum method, quasi-nonlocal coupling, finite
  deformations, sharp stability, error analysis}

\begin{abstract}
  For a next-nearest neighbour pair interaction model in a periodic
  domain, {\it a priori} and {\it a posteriori} analyses of the
  quasinonlocal quasicontinuum method (QNL-QC) are presented. The
  results are valid for large deformations and essentially guarantee a
  one-to-one correspondence between atomistic solutions and QNL-QC
  solutions. The analysis is based on truncation error and residual
  estimates in negative norms and novel {\it a priori} and {\it a
    posteriori} stability estimates.
\end{abstract}

\maketitle

\section{Introduction}
\label{sec:intro}
Quasicontinuum (QC) methods
\cite{E:2006,Miller:2008,Miller:2003a,Ortiz:1995a,Shimokawa:2004} are
a class of prototypical schemes for coupling atomistic models of
solids to continuum elasticity models. They combine, in principle, the
accuracy of atomistic models for defects with the efficiency of
continuum models for long-range elastic effects. For further detail on
the historical development of the QC method see
\cite{Miller:2008, Miller:2003a}.

Not all QC methods are equally reliable in modelling large atomistic
systems. Recent analyses of QC methods
\cite{legoll,Dobson:2008b,Dobson:qcf.stab,Dobson:arXiv0903.0610,Dobson:qce.stab,Gunzburger:2008a,Luskin:clusterqc,emingyang}
have identified various sources of inconsistencies (in the sense of
numerical analysis; see also Section \ref{sec:consistency}) and
instabilities due to inadequate treatments of the interface. The two
methods that have emerged as the most promising candidates are the
force-based QC method \cite{Dobson:2008a} and the quasinonlocal QC
(QNL-QC) method \cite{Shimokawa:2004} (and its generalization
\cite{E:2006} to longer interaction ranges).

All contributions to the analysis of QC methods cited above consider
only linearized problems, except for \cite{emingyang} where small
deformations from a uniform reference state are admitted. The main
purpose of the present paper is to introduce a simple yet powerful
analytical framework that allows a generalization of these results to
genuinely nonlinear situations.  To present the new ideas that are
required to achieve this in the simplest possible setting, we will
focus on a one-dimensional model problem where the atomistic model is
formulated in terms of a second neighbour pair interaction energy with
applied dead loads. For this model, the QNL-QC method of Shimokawa
{\it et al.}  \cite{Shimokawa:2004} was previously analyzed in
\cite{Dobson:2008b,emingyang}. However, the techniques used there
would be difficult to extend to the large deformation regime, for
which an entirely new approach will be presented here, based on
consistency error estimates in negative Sobolev norms and sharp {\it a
  priori} and {\it a posteriori} stability estimates. Moreover, care
will be taken throughout this paper that the techniques can, in
principle, be extended to higher dimensions and to situations with
defects. This will primarily be achieved by modifying ideas from
\cite{Ortner:2008a} (where a Galerkin projection without continuum
approximation is analyzed), so that no higher regularity results for
the hessian operator of the atomistic or QC energy functionals are
required.

The present paper also develops the first residual-based {\it a
  posteriori} error analysis for the QNL-QC method.  Theorem
\ref{th:exapost} is the first {\it a posteriori} error estimate for
the QNL-QC method which rigorously establishes the existence of exact
solutions for which the estimate applies. The application of goal
oriented {\it a posteriori} error control to QC methods was pioneered
in \cite{arndtluskin07b,arndtluskin07c,prud06}.

At several places in the paper interesting and challenging open
problems are discussed.

Finally, it should be remarked that the purpose of this paper is to
present a new framework for the analysis of the QC method in the
simplest possible setting that still allows the main features one
observes in applications: non-smooth solutions and large
deformations. Numerical experiments illustrating the results presented
here, as well as further interesting observations will be presented
after introducing finite element coarse graining into the framework
\cite{OrtnerWang:2009a}.

\subsection*{Outline}
We begin, in Section \ref{sec:intro:model_problem}, by presenting an
atomistic model problem in a periodic formulation that avoids the
difficulties posed by boundaries. In Section \ref{sec:model:local_QC},
we formulate the QNL-QC approximation {\em without} coarse
graining. For an analysis of the QNL-QC including finite element
coarse graining of the continuum region see
\cite{OrtnerWang:2009a}. In the remainder of section \ref{sec:model}
we discuss some features of the atomistic model and introduce the
necessary technical background for the subsequent analysis.

In Section \ref{sec:consistency} we discuss the concepts of {\em
  truncation error} and {\em residual} in the context of the QC
method. Previous work \cite{Dobson:2008b} has analyzed the truncation
error in weighted $\ell^p$-norms, which leads to suboptimal truncation
error estimates. As a result, somewhat subtle and technically
demanding ideas were required to still obtain (quasi-)optimal error
estimates. By contrast, we will see here that if the ``correct'' norm
is chosen, namely a negative Sobolev norm, then these complications
can be completely avoided. In \cite{emingyang}, an interesting variant
of a negative norm was used that is useful for the nonlinear analysis
but does not lead to (quasi-)optimal error estimates (see Remarks
\ref{rem:spijker} and \ref{rem:apriori}).

The subject of Section \ref{sec:stability} is a careful and general
stability analysis of the QNL-QC method. First, an idea from
\cite{Dobson:qce.stab} will be extended to obtain {\it a priori}
stability results for large deformations. The main novel contribution
in this section, however, is the {\it a posteriori} stability result
in Theorem \ref{th:apost_stab}. This section concludes with a brief
discussion of the difficulties encountered when defects are present,
and how the results might be generalized.

Finally, in Section \ref{sec:existence} we combine the consistency and
stability analyses of the previous sections to obtain {\it a priori}
as well as {\it a posteriori} existence results and error
estimates. The techniques used here are not too dissimilar from
\cite{emingyang}, however, the point of view taken is a very different
one and leads to a different set of existence results. Moreover, the
sharper consistency and stability results of the present paper lead to
error estimates that are truly (quasi-)optimal in the atomistic
spacing and the smoothness of the atomistic solution. See Remarks
\ref{rem:apriori} and \ref{rem:aposteriori} for more detail.

\section{Atomistic Model and QNL-QC Approximation}
\label{sec:model}

\subsection{An atomistic model problem}
\label{sec:intro:model_problem}
In atomistic models surfaces create boundary layers in the deformation
field and must therefore be considered defects in the lattice. We
avoid this issue, by working in an infinite lattice $\eps \Z$, where
$\eps > 0$ is the reference {\em lattice spacing}. Since the
functional analysis becomes unnecessarily difficult in this infinite
domain we will, moreover, restrict the set of admissible deformations
to those which are $N$-periodic displacements from the reference
lattice, that is, we define
\begin{align*}
  \Ys =~& F x + \Us, \\
  \intertext{where $x_\xi = \eps \xi$ for $\xi \in \Z$, and $F > 0$ is
    a {\em macroscopic deformation gradient}, and where}
  \Us =~& \big\{ u \in \R^\Z : u_{\xi+N} = u_\xi \quad \forall \xi \in
  \Z, \text{ and } {\textstyle \sum_{\xi = 1}^N} u_\xi = 0 \big\}.
\end{align*}
The condition $\sum_{\xi = 1}^N u_\xi = 0$ is an artefact of the
periodic boundary conditions, and ensures locally unique solvabililty
of the equilibrium equations to the models that are introduced
below. We set $\eps = 1/N$ throughout so that the reference length of
the period is one.

We denote the algebraic dual of $\Us$ by $\Us^*$. Equipped with the
weighted $\ell^2$-inner product
\begin{displaymath}
  \< v, w \>  = \eps \sum_{\xi = 1}^N v_\xi w_\xi,
\end{displaymath}
$\Us$ becomes a Hilbert space, and hence, any $T \in \Us^*$ can be
represented by a function $z_T \in \Us$ via $T[v] = \< z_T, v \>$ for
$v \in \Us$. Thus, even though forces are usually best understood as
elements of $\Us^*$, we will usually identify them with ``functions''
$f \in \Us$.

The {\em stored energy per period} of a deformation $y \in \Ys$ is
given by a next-nearest neighbour pair interaction model,
\begin{equation}
  \label{eq:a_defn}
  \Phi(y) := \eps \sum_{\xi = 1}^N \big(\phi(y_\xi')
  + \phi(y_\xi' + y_{\xi+1}')\big),
\end{equation}
where $v_\xi' = \eps^{-1} (v_\xi - v_{\xi-1})$ for $v \in \R^\Z$,
and where $\phi$ is a Lennard-Jones type interaction potential which
satisfies the following properties:
\begin{enumerate}
\item[($\phi 1$)] \quad $\phi \in \CC^3((0, +\infty]; \R)$, and
\item[($\phi 2$)] \quad there exists $r_* > 0$ such that $\phi$ is convex
  in $(0, r_*)$ and concave in $(r_*, +\infty)$.
\end{enumerate}
By $\phi \in \CC^3((0, +\infty]; \R)$ we mean that $\phi$ and its
first three derivatives are bounded in any interval $(\delta,
+\infty)$, $\delta > 0$. Implicitly, we actually assume that $\phi(r)$
and its derivatives decay rapidly to zero as $r \to \infty$. This
justifies the next-nearest neighbour interaction model.

We assume, for simplicity, that the atomistic system is subjected to a
dead load $f \in \Us$, so that the total energy of a deformation $y
\in \Ys$ is given by
\begin{displaymath}
  \Phi^\tot(y) = \Phi(y) - \< f, y\>.
\end{displaymath}
Our goal is to solve the minimization problem
\begin{equation}
  \label{eq:minproblem_a}
  y \in \argmin \Phi^\tot(\Ys),
\end{equation}
where $\argmin$ denotes the set of local minimizers. The first order
necessary optimality condition (or, Euler--Lagrange equation) for
\eqref{eq:minproblem_a} is
\begin{equation}
  \label{eq:model:crit_a}
  D\Phi^\tot(y)[u] = 0 \qquad \forall u \in \Us.
\end{equation}
A deformation $y \in \Ys$ satisfying \eqref{eq:model:crit_a} is called
an {\em equilibrium} (or critical point). If an equilibrium $y$ also
satisfies the {\em sufficient second order optimality condition}
\begin{equation}
  \label{eq:1}
  D^2 \Phi^\tot(y)[u,u] > 0 \qquad \forall u \in \Us \setminus \{0\},
\end{equation}
then we say that $y$ is a {\em strongly stable equilibrium}. A
strongly stable equilibrium is an isolated local minimizer of
$\Phi^\tot$ in $\Ys$.

\subsection{The quasi-nonlocal QC method}
\label{sec:model:local_QC}
If a deformation $y \in \Ys$ is smooth, in the sense that $y_\xi'
\approx y_{\xi+1}'$ for all $\xi$, then we can approximate the
second-neighbour terms by
\begin{equation}
  \label{eq:qnl_nnn_splitting}
  \phi(y_\xi' + y_{\xi+1}') \approx 
  \smfrac12 \big\{ \phi(2y_\xi') + \phi(2y_{\xi+1}') \big\},
\end{equation}
which leads to an approximate stored energy functional
\begin{displaymath}
  \Phi_\cb(y) = \eps \sum_{\xi = 1}^N \big\{ \phi(y_\xi')
  + \phi(2 y_\xi') \big\}
  = \eps \sum_{\xi = 1}^N \phi_\cb(y_\xi'),
\end{displaymath}
where $\phi_\cb(r) = \phi(r) + \phi(2r)$ is called the Cauchy--Born
stored energy function. The Cauchy--Born stored energy functional
$\Phi_\cb$ is {\em local} and its minimization can be achieved
efficiently by means of finite element methodology. (Note, though,
that in 1D this is not really an issue \cite{Ortner:2008a}.) Moreover,
it is considered an excellent approximation to the atomistic model in
a smooth deformation regime (see \cite{E:2007a} for a rigorous theory
of the small deformation regime).

If a minimizer of the atomistic model, which we are trying to compute,
has defects in a small localized region but is smooth elsewhere then
we should couple the atomistic model to the continuum Cauchy--Born
model. The QNL-QC method \cite{Shimokawa:2004} is the simplest
energy-based coupling method that is consistent for next-nearest
neighbour models (see Section \ref{sec:consistency} for a discussion
of consistency). For the simple pair interaction energy we use in this
paper it can be defined as follows.

Let $\As \subset \{1, \dots, N\}$ be the {\em atomistic region} and
let $\Cs = \{ 1, \dots, N\} \setminus \As$ be the {\em continuum
  region}. The nearest neigbour terms are left unchanged, while, for
$\xi \in \As$, the next-nearest neighbour interaction $\phi(y_\xi' +
y_{\xi+1}')$ is approximated as in \eqref{eq:qnl_nnn_splitting}. This
procedure leads to the QNL-QC stored energy functional
\begin{equation}
  \label{eq:model:qnl_fcnl}
  \Phi_\h(y) = \eps \sum_{\xi = 1}^N \phi(y_\xi')
  + \eps \sum_{\xi \in \As} \phi(y_\xi' + y_{\xi+1}')
  + \eps \sum_{\xi \in \Cs} \smfrac12 \big\{ \phi(2 y_\xi') 
    + \phi(2 y_{\xi+1}') \big\}.
\end{equation}
Upon defining the total energy
\begin{displaymath}
  \Phi_\h^\tot(y) = \Phi_\h(y) - \< f, y \>,
\end{displaymath}
in the QNL-QC method, we solve the minimization problem
\begin{equation}
  \label{eq:qnl_minproblem}
  y_\h \in \argmin \, \Phi_\h^\tot(\Ys).
\end{equation}
The associated first order optimality condition in variational form is
\begin{equation}
  \label{eq:crit_qc}
  D\Phi_\h^\tot(y_\h)[u] = 0  \qquad \forall u \in \Us.
\end{equation}

\begin{remark}
  For practical reasons, it is important that this approximation can
  be rewritten in the form
  \begin{displaymath}
    \Phi_\h(y) = \Phi_\h^\As(y) + \Phi_\h^\Is(y) + \Phi_\h^\Cs(y),
  \end{displaymath}
  where $\Phi_\h^\As$, $\Phi_\h^\Is$, and $\Phi_\h^\Cs$ are the
  contributions from the atomistic region, the interfacial region, and
  the continuum region, respectively, and so that $\Phi_\h^\As$ is
  precisely the atomistic model restricted to (a subset of) $\As$,
  that $\Phi_\h^\Cs$ is precisely the Cauchy--Born approximation
  restricted to (a subset of) $\Cs$, and $\Phi_\h^\Is$ is an interface
  correction that is cheap to compute. This formulation allows for
  finite element coarsening of the continuum region. For more details
  on this splitting and its importance, see
  \cite{OrtnerWang:2009a}. For the purposes of the analysis presented
  in this paper, however, \eqref{eq:model:qnl_fcnl} is the most useful
  formulation.
\end{remark}

\begin{figure}
  \centering
  \includegraphics[width=9cm]{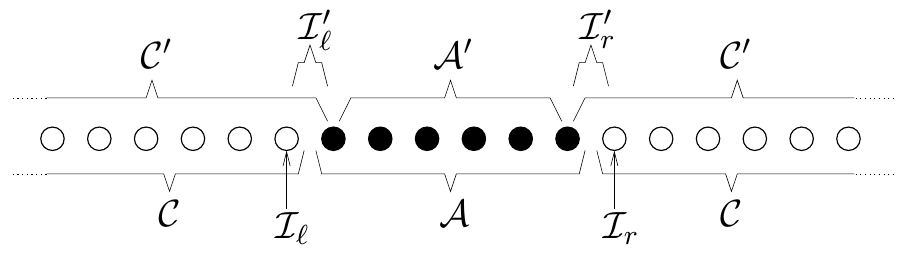}
  \caption{\label{fig:ACI_sets} Visualization of the atomistic region,
    the continuum region, the interface sets, as well as their
    counterparts for bounds, defined in Section
    \ref{sec:model:local_QC}.}
\end{figure}

For future reference, we define the left- and right-interface sets
({\em left} and {\em right} are understood in relation to the
atomistic region(s)) as follows,
\begin{equation}
  \label{eq:model:defn_I}
  \Is_\ell = \big\{ \xi \in \Cs : \xi + 1 \in \Asper \big\}
  \quad \text{and} \quad
  \Is_r = \big\{ \xi \in \Cs : \xi - 1 \in \Asper \big\},
\end{equation}
where $\Asper$ denotes the periodically extended atomistic region,
\begin{displaymath}
  \Asper = \bigcup_{j \in \Z} \big(Nj + \As).
\end{displaymath}
The union $\Is = \Is_\ell \cup \Is_r$ is simply called the {\em
  interface}. A visualization of these, and the following, definitions
is shown in Figure \ref{fig:ACI_sets}.

We assume throughout that
\begin{equation}
  \label{eq:I_dontintersect}
  \Is_\ell \cap \Is_r = \emptyset.
\end{equation}
This means, that any connected component of the continuum region must
contain at least two atoms.

In addition, since we will often write expressions in terms of bonds
rather than atoms, it is sometimes more convenient to use the
following variants of these sets:
\begin{displaymath}
  \Is_\ell' = \Is_\ell+1, \quad \Is_r' = \Is_r, \quad
  \As' = \As \setminus \Is_\ell', \quad \Cs' = \Cs \cup \Is_\ell',
  \quad \text{and} \quad \Is' = \Is_r' \cup \Is_\ell'.
\end{displaymath}
These definitions are made mostly for the sake of an intuitive
notation in the main results. On a first reading, it may be best to
ignore these subtle differences and simply bear in mind that the sets
$\As'$, etc., are variants of the originals that allow for an
attractive presentation.


\subsection{A bound on next-nearest neighbour interactions}
\label{sec:bound_nnn}
It will be a crucial ingredient in the analysis of the QNL-QC method
to assume that next-nearest neighbour interactions remain in the
concave region of the interaction potential, that is, we assume that
the solutions of the atomistic model and of the QNL-QC method satisfy
\begin{equation}
  \label{eq:n2_interaction_bound}
  y_\xi' \geq r_* / 2 \qquad \forall \xi \in \{1, \dots, N\}.
\end{equation}
For most interesting interaction potentials, the case $y_\xi' < r_* /
2$ can only be achieved under extreme compressive forces that will not
usually be observed in experiments. As a matter of fact, one can
safely assume that under such extreme conditions a classical potential
(as opposed to a potential based on quantum mechanics) and zero
temperature statics constitute an inappropriate model to begin with.

Note, moreover, that the condition $y_\xi' \geq r_* / 2$ also ensures
that $\phi$ and its derivatives satisfy uniform bounds. For future
reference, we define the constants
\begin{displaymath}
  C_j(s) := \sup_{r \geq s} \phi^{(j)}(r) \qquad \text{for } s > 0
  \text{ and for } k = 0,1, \dots,
\end{displaymath}
where $\phi^{(j)}$ denotes the $j$th derivative of $\phi$.

\subsection{Further notation and an auxiliary result}
\label{sec:notation}
The second and third finite differences are defined, for $v \in
\R^\Z$, by
\begin{align*}
  v_\xi'' :=~& \eps^{-2}(v_{\xi+1} - 2 v_\xi + v_{\xi-1}), \quad \text{and} \\
  v_\xi''':=~& \eps^{-3}(v_{\xi+1}-3 v_{\xi} + 3 v_{\xi-1} - v_{\xi-2}).
\end{align*}
For $v \in \Us$, the functions $v'$, $v''$, and $v'''$ are understood
as $N$-periodic functions. We normally associate the function value
$v_\xi$ and the second difference $v_\xi''$ with the atom $\xi$, while
the first difference $v_\xi'$ and the third difference $v_\xi'''$ are
normally associated with the cell, or bond, $(\xi-1, \xi)$.

In addition to the weighted $\ell^2_\eps$-inner product we define the
weighted $\ell^p$-norms
\begin{displaymath}
\| v \|_{\ell^p_\eps} := \cases{
  \big( \eps\sum_{\xi=1}^N |v_\xi|^p \big)^{1/p}, & 1 \leq p < \infty, \\[3pt]
  \displaystyle \max_{\xi = 1, \dots, N} |v_\xi|, & p = \infty,
  }
\end{displaymath}
We normally drop the subscript in $\ell^\infty_\eps$ and write
$\ell^\infty$ instead..

The space $\Us$ will often be equipped with the Sobolev-like norms
\begin{displaymath}
  \| v \|_{\Us^{1,p}} = \|v'\|_{\ell^p_\eps}, \quad
  \text{for } v \in \Us \quad \text{ and } p \in [1,\infty].
\end{displaymath}
The space $\Us$ equipped with the $\Us^{1,p}$-norm is then simply
denoted $\Us^{1,p}$. For $p' = p / (p-1)$, the norm on the topological
dual $\Us^{-1,p} := (\Us^{1,p'})^*$ of $\Us^{1,p'}$ is denoted
\begin{displaymath}
  \| T \|_{\Us^{-1,p}} := \| T \|_{(\Us^{1,p'})^*} = 
  \sup_{\substack{ v \in \Us \\ \|v \|_{\Us^{1,p'}} = 1 }} T[v],
  \qquad \text{for } T \in \Us^{-1,p}.
\end{displaymath}

If $\Psi : \Ys \to \R$ is (Fr\'{e}chet) differentiable then its first
variation is denoted $D\Psi$ and is understood as a nonlinear map from
$\Ys$ to $\Us^*$, that is, for $y \in \Ys$, $D\Psi(y) \in \Us^*$, $v
\mapsto D\Psi(y)[v]$. Similarly, the second variation $D^2\Psi$ is
understood as a nonlinear map from $\Ys$ to $L(\Us, \Us^*)$, that is,
for $y \in \Ys$, $D^2\Psi(y) \in L(\Us, \Us^*)$. Equivalently,
$D^2\Psi(y)$ can be understood as a symmetric bilinear form on $\Us
\times \Us$, $(v, w) \mapsto D^2\Psi(y)[v,w]$. We will henceforth make
no distinction between these different interpretations and use
whichever is most convenient at any given moment.

\begin{lemma}[Inverse Function Theorem]
  \label{th:inverse_fcn_thm}
  Let $X, Y$ be Banach spaces, $A$ an open subset of $X$, and let $\Fs :
  A \to Y$ be Fr\'{e}chet differentiable. Suppose that $x_0 \in A$
  satisfies the conditions
  \begin{align*}
    &  \| \Fs(x_0) \|_{Y} \leq \eta, \quad
    \| D\Fs(x_0)^{-1} \|_{L(Y,X)} \leq \sigma, \\
    & \overline{B_X(x_0, 2\eta\sigma)} \subset A, \\
    & \| D\Fs(x_1) - D\Fs(x_2) \|_{L(X,Y)} \leq L \| x_1 - x_2 \|_X \quad
    \text{for} \quad \|x_j - x_0\|_X \leq 2 \eta \sigma, \\
    & \text{ and } 2 L \sigma^{2} \eta < 1,
  \end{align*}
  then there exists $x \in X$ such that $\Fs(x) = 0$ and $\|x -
  x_0\|_X \leq 2\eta\sigma$.
\end{lemma}
\begin{proof}
  The result follows, for example, by applying Theorem 2.1 in
  \cite{ortner_apostex} with the choices $R = 2 \eta\sigma$,
  $\omega(x_0, R) = LR$ and $\bar\omega(x_0, R) = \frac12 L
  R^2$. Similar results can be obtained by tracking the constants in
  most proofs of the inverse function theorem, and assuming local
  Lipschitz continuity of $D\Fs$.
\end{proof}

\section{Consistency} 
\label{sec:consistency}
It was pointed out in \cite{Dobson:2008b} that the QNL-QC method is
{\em not} consistent. This observation was based on the fact that the
authors computed the consistency error in the $\ell^\infty$-norm and
considered the limit as $\eps \to 0$. As a matter of fact, the QNL-QC
(and even the original energy-based QC method \cite{Ortiz:1995a}) is
consistent in appropriate negative norms (see also \cite[Lemma
5.11]{emingyang}, and \cite{OrtnerWang:2009a} for more general
discussion). However, before we prove this result, we briefly review
the notions of truncation error and residual, and explain more clearly
what we mean by consistency.

Let $y \in \argmin\,\Phi^\tot(\Ys)$, then the {\em truncation error}
(associated with this solution) is the linear functional $T \in
\Us^*$, which measures the extent to which $y$ fails to
satisfy~\eqref{eq:model:crit_a}:
\begin{displaymath}
  T[u] = D\Phi_\h^\tot(y)[u] = D\Phi_\h(y)[u] - \< f, u \>.
\end{displaymath}
Conversely, if $y_\h \in \argmin \Phi^\tot_\h(\Ys)$ then the linear
functional $R \in \Us^*$ that measures the extent to which $y_\h$
fails to satisfy (\ref{eq:model:crit_a}) is called the {\em residual}
(of the approximate solution $y_\h$):
\begin{displaymath}
  R[u] = D\Phi^\tot(y_\h)[u] = D\Phi(y_\h)[u] - \< f, u \>. 
\end{displaymath}
Since $y$ satisfies (\ref{eq:model:crit_a}) and $y_\h$ satisfies
(\ref{eq:crit_qc}) we deduce that
\begin{align*}
  T = D\Phi_\h(y) - D\Phi(y), \quad \text{and} \quad
  R = D\Phi(y_\h) - D\Phi_\h(y_\h).
\end{align*}
Thus, estimating the truncation error will automatically give us a
residual estimate and vice-versa. Hence, we will call both $R$ and $T$
simply the {\em consistency error}. (This observation ceases to be
valid if coarse-graining is taken into account.)

Traditionally, we call a numerical method {\em consistent} if the
truncation error tends to zero, in a suitable sense, as the mesh size,
or another discretization parameter, tends to zero. In the present
case, since $\eps$ is fixed (it is best to think of $\eps$ to be in of
the order $O(10^{-3})$ or $O(10^{-4})$), we cannot make such a
definition. We will therefore use the term rather loosely and simply
discuss the {\em order of consistency} of the method. For example
first order consistency would mean that the consistency error can be
bounded by $\eps$ and a factor that depends on the smoothness of
$y$. However, even this nomenclature is not wholly appropriate as the
QNL-QC method has different orders in different parts of the
domain. This is demonstrated in the following theorem, which provides
a sharp bound on the consistency error of the QNL-QC method in the
$\Us^{-1,p}$-norms, $1 \leq p \leq \infty$. Even though we will later
only use $\Us^{-1,2}$ estimates, the more general case is included
here since it requires no additional work.

\begin{theorem}[Consistency in $\Us^{1,p}$]
  \label{th:consistency}
  Let $y \in \Ys$ such that $\min_{\xi \in \Z} y_\xi' > 0$, then
  \begin{displaymath}
    \| D\Phi(y) - D\Phi_\h(y) \|_{\Us^{-1,p}} \leq 
    \eps \bar{C}_2 \|y''\|_{\ell^p_\eps(\Is)} + \eps^2 \bar{C}_3 \big(
    \|y'''\|_{\ell^p(\Cs'\setminus\Is')} + \| y'' \|_{\ell^{2p}_\eps(\Cs)}^2 \big),
  \end{displaymath}
  where $\bar{C}_i = C_i(2\min_{\xi\in\Cs'} y_\xi')$, $i = 2, 3$.
\end{theorem}
\begin{proof}
  The first variations of the atomistic and QNL-QC functionals at $y$
  are, respectively, given by
  \begin{align}
    \label{eq:DPhi}
    D\Phi(y)[u] = \eps \sum_{\xi = 1}^N \phi'(y_\xi')u_\xi'
    ~+~& \eps \sum_{\xi = 1}^N \phi'(y_\xi' + y_{\xi+1}')[u_\xi' + u_{\xi+1}'], \\
    \intertext{and}
    \label{eq:DPhih}
  D\Phi_\h(y)[u] = \eps \sum_{\xi = 1}^N \phi'(y_\xi')u_\xi'
  ~+~& \eps \sum_{\xi \in \As} \phi'(y_\xi' + y_{\xi+1}')[u_\xi'+u_{\xi+1}'] \\
  \notag
  ~+~& \eps \sum_{\xi \in \Cs} \smfrac12 \big\{ 
  2 \phi'(2 y_\xi') u_\xi' + 2 \phi'(2y_{\xi+1}') u_{\xi+1}' \big\}.
  \end{align}
  The consistency error $T := D\Phi(y) - D\Phi_\h(y)$ therefore
  satisfies
  \begin{equation}
    \label{eq:truncation_error}
    T[u] = \eps \sum_{\xi \in \Cs} \big\{ 
    \phi'(y_\xi' + y_{\xi+1}')[u_\xi'+u_{\xi+1}']
    - \phi'(2 y_\xi') u_\xi' - \phi'(2 y_{\xi+1}') u_{\xi+1}' \big\}
  \end{equation}
  for all $u \in \Us$.  From (\ref{eq:truncation_error}) we can
  immediately obtain a first-order consistency error estimate,
  however, one can improve upon this by taking into account symmetries
  of the interaction law.

  If we collect all coefficients that premultiply a term $u_\xi'$,
  then we obtain second-order errors in the ``interior'' of the
  continuum region and first-order errors in the interface. To see
  this, we recall the definitions of the interface regions from
  Section \ref{sec:model:local_QC}, and compute
  \begin{align}
    \notag
    T[u] =~& \eps \sum_{\xi \in \Is_\ell} \big\{ \phi'(y_{\xi}'+y_{\xi+1}') 
    - \phi'(2y_{\xi+1}') \big\} u_{\xi+1}'
    + \eps \sum_{\xi \in \Is_r} \big\{ \phi'(y_\xi'+y_{\xi+1}') 
    - \phi'(2y_\xi') \big\} u_\xi' \\
    \label{eq:cons:T_2}
    & + \eps\sum_{\xi \in \Cs' \setminus \Is'} 
    \big\{  \phi'(y_{\xi-1}'+y_\xi') + \phi'(y_\xi' + y_{\xi+1}') 
    - 2 \phi'(2y_\xi') \big\} u_\xi'.
  \end{align}

  Taylor expansions of the first and second terms in curly brackets
  yield
  \begin{equation}
    \label{eq:cons:taylor_1}
    \begin{split}
      \phi'(y_\xi'+y_{\xi+1}') - \phi'(2y_{\xi+1}') =~&
      - \eps \phi''(2 y_{\xi+1}') y_\xi''
      + \smfrac12 \eps^2 \phi'''(\vartheta_{1,\xi}) |y_\xi''|^2,
      \quad \text{and} \\
      \phi'(y_\xi'+y_{\xi+1}') - \phi'(2y_{\xi}') =~& 
      \eps \phi''(2 y_{\xi}') y_\xi''
      + \smfrac12 \eps^2 \phi'''(\vartheta_{2,\xi}) |y_\xi''|^2,  
    \end{split}
  \end{equation}
  where $\vartheta_{1,\xi} \in \conv\{ y_\xi' + y_{\xi+1}', 2
  y_{\xi+1}' \}$ and $\vartheta_{2,\xi} \in \conv\{ y_\xi' +
  y_{\xi+1}', 2 y_{\xi}' \}$, and hence $|\phi'''(\vartheta_{i,\xi})|
  \leq C_3(2 \min_{\xi\in\Cs'} y_\xi') =: \bar{C}_3$ for $i = 1, 2$.

  To expand the third term in curly brackets in (\ref{eq:cons:T_2}),
  we use (\ref{eq:cons:taylor_1}) twice to obtain
  \begin{align*}
    & \phi'(y_\xi' + y_{\xi+1}') + \phi'(y_{\xi-1}'+y_\xi') - 2\phi'(2 y_\xi') \\
    =~& \eps^2 \phi'''(2 y_\xi') y_\xi'''
    +  \eps^2 \smfrac12 \big\{ \phi'''(\vartheta_{2,\xi}) |y_\xi''|^2
    + \phi'''(\vartheta_{1,\xi-1}) |y_{\xi-1}''|^2 \big\}.
  \end{align*}

  Inserting these expansions and bounds into (\ref{eq:cons:T_2}), and
  applying several weighted H\"{o}lder inqualities we can estimate
  \begin{align*}
    \big|T[u]\big| 
    \leq~& \eps \sum_{\xi \in \Is_\ell} \big\{ \eps \bar{C}_2 |y_\xi''|
    + \eps^2 \smfrac12 \bar{C}_3 |y_\xi''|^2 \big\} |u_{\xi+1}'|
    + \eps \sum_{\xi \in \Is_r} \big\{ \eps \bar{C}_2 |y_\xi''|
    + \eps^2 \smfrac12 \bar{C}_3 |y_\xi''|^2 \big\} |u_{\xi}'| \\
    & + \eps \sum_{\xi \in \Cs \setminus \Is_r} \big\{ \
    \eps^2 \bar{C}_3 |y_\xi'''| + \eps^2 \smfrac12 \bar{C}_3 \big( 
    |y_\xi''|^2 + |y_{\xi-1}''|^2 \big) \big\} |u_\xi'| \\
    \leq~& \big\{\eps \bar{C}_2 \| y'' \|_{\ell^p_\eps(\Is)} 
    + \eps^2 \bar{C}_3 \| y''' \|_{\ell^p_\eps(\Cs\setminus\Is_r)}
    + \eps^2 \bar{C}_3 \| y'' \|_{\ell^{2p}_\eps(\Cs)}^2 \big\} 
    \|u'\|_{\ell^{p'}_\eps}.
    \qedhere
  \end{align*}
\end{proof}

\begin{remark}
  The consistency error has no contribution from the atomistic region,
  which is not surprising since the model is exact in $\As$. In the
  continuum region, the consistency error reflects the second-order
  accuracy of the Cauchy--Born approximation for simple lattices
  \cite{E:2007a}. Finally, in the interface, the consistency error is
  only of first order. By estimating $\|y''\|_{\ell^p_\eps(\Is)}$
  above by $\|y''\|_{\ell^\infty(\Is)}$, we can however obtain
  additional powers of $\eps$,
  \begin{equation}
    \label{eq:higher_epsn_I}
    \eps \bar{C}_2 \|y''\|_{\ell^p_\eps(\Is)} \leq
    \eps^{1+1/p} \bar{C}_2 (\# \Is)^{1/p} \|y''\|_{\ell^\infty(\Is)},
  \end{equation}
  but at the cost of a dependency on the size of the interfacial
  region. Moreover it should be stressed that such an estimate might
  hide the significantly reduced accuracy of the QNL-QC method in the
  interface. 

  This loss of accuracy at the interface can potentially become
  problematic if very high accuracy is sought. For example, a P2
  finite element discretization of the continuum region will not be
  able to increase the accuracy of the computation beyond the error
  introduced in this interface. Thus, a higher order coupling
  mechanism would be a highly desirable goal.

  A finer analysis of this reduced accuracy for a linearized model
  problem can be found in \cite{Dobson:2008b}.
\end{remark}

\begin{remark}
  Generalizations of Theorem \ref{th:consistency} for a variety of QC
  methods can be found in \cite{OrtnerWang:2009a}. There, we also
  include finite element coarse graining in the analysis. We use
  negative norm consistency error estimates, similar to Theorem
  \ref{th:consistency}, to control what we call the {\em model error}.
\end{remark}

\begin{remark}
  \label{rem:spijker}
  In \cite[Lemma 5.11]{emingyang} a negative-norm truncation error
  estimate with respect to a so-called {\em Spijker norm} can be
  found, which, in the language of the present work, reads
  \begin{displaymath}
    \big| D\Phi(y)[u] - D\Phi_\h(y)[u] \big|
    \leq C \eps \big( \eps^{-1/2} \| u' \|_{\ell^2_\eps} \big).
  \end{displaymath}
  This estimate follows immediately from \eqref{eq:higher_epsn_I}. The
  reason for choosing this norm is that it guarantees uniform
  closeness of gradients (i.e., $\eps^{-1/2} \| u' \|_{\ell^2_\eps}
  \geq \|u'\|_{\ell^\infty}$), a fact that is highly useful in the
  nonlinear analysis. Unfortunately it leads to suboptimal error
  estimates. In the analysis of Section \ref{sec:existence}, we will
  circumvent this problem by using inverse estimates, but still retain
  the optimal truncation error estimates.
\end{remark}

\section{Stability}
\label{sec:stability}
Having established a fairly sharp consistency error estimate, we turn
to the question of stability of the QNL-QC method. Stability estimates
are usually easiest to establish in a Hilbert space, particularly,
when the operator under investigation is coercive. Coercivity of the
QNL-QC hessian evaluated at the reference state $y = Fx$ was
established in \cite{Dobson:2008b,emingyang}, and it was shown in
\cite{Dobson:qce.stab} that these estimates sharply reflect the
stability of the full atomistic model. In the following sections we
will generalize these results to various situations that admit large
deformations.

In one dimension, it is still relatively straightforward to derive
stability bounds in the space $\Us^{1,p}$. For example, several
stability results for atomistic and quasicontinuum models in the space
$\Us^{1,\infty}$, which are particularly useful for a nonlinear
analysis, were derived in \cite{Ortner:2008a}. However, since such
results would be very difficult (if not impossible) to obtain in
2D/3D, we will attempt in the next sections to use only
$\Us^{1,2}$-stability results in our analysis and still obtain
similarly strong nonlinear results.

\subsection{Preliminary remarks}
\label{sec:stab}
{\em Linear stability} for minimization problems is normally
associated with the coercivity constant (or smallest eigenvalue) in an
appropriate function space. Thus, to relate the stability of the
atomistic model and that of the QNL-QC, we would like to prove a
result of the type
\begin{equation}
  \label{eq:approx_stab}
  c(y) := \inf_{ \substack{u \in \Us \\\|u\|_{\Us^{1,2}} = 1  }} D^2\Phi(y)[u,u]
   \approx \inf_{ \substack{u \in \Us \\\|u\|_{\Us^{1,2}} = 1  }} D^2\Phi_\h(y)[u,u] 
   =: c_\h(y).
\end{equation}
If the approximate hessian $D^2\Phi_\h$ could be obtained by
perturbing the coefficients of $D^2\Phi$ then one could obtain such a
result from an error estimate on the hessians, that is, from a bound on
$\| D^2\Phi(y) - D^2 \Phi_\h(y) \|_{L(\Us^{1,2},
  \Us^{-1,2})}$. However, if we compute the two hessians explicitly,
\begin{align*}
  D^2\Phi(y)[u,u] = \eps\sum_{\xi = 1}^N \phi''(y_\xi')|u_\xi'|^2
  ~&+~ \eps \sum_{\xi = 1}^N \phi''(y_\xi' + y_{\xi+1}') |u_\xi' + u_\xi'|^2,
  \quad \text{and} \\
  D^2\Phi_\h(y)[u,u] = \eps\sum_{\xi = 1}^N \phi''(y_\xi')|u_\xi'|^2
  ~&+~ \eps \sum_{\xi \in \As} \phi''(y_\xi' + y_{\xi+1}') |u_\xi' + u_\xi'|^2 \\
  &+~ \eps \sum_{\xi \in \Cs} \smfrac12 \big\{ \phi''(2y_\xi') |2u_\xi'|^2
  +~ \phi''(2y_{\xi+1}') |2u_{\xi+1}'|^2 \big\},
\end{align*}
we observe that in the continuum region $D^2\Phi(y)[u,u]$ contains
non-zero coefficients for the mixed terms $u_\xi' u_{\xi+1}'$ which
are replaced by diagonal terms in the QNL-QC Hessian. This shows that
it is impossible to obtain a useful estimate on the difference of
the Hessians. Instead, we must aim to prove (\ref{eq:approx_stab})
directly.

The crucial observation which is the basis of the stability analysis
in this section is that the non-local hessian terms $|u_\xi' +
u_{\xi+1}'|^2$ can be rewritten in terms of the local quantities
$|u_\xi'|^2$ and $|u_{\xi+1}'|^2$ and a strain-gradient correction,
\begin{equation}
  \label{eq:stab:straingrad}
  |u_\xi' + u_{\xi+1}'|^2 = 2 |u_\xi'|^2 + 2 |u_{\xi+1}'|^2
  - \eps^2 |u_\xi''|^2.
\end{equation}
Using this simple formula it is straightforward to rewrite the
Hessians in the form
\begin{equation}
  \label{eq:stab:hessians}
  \begin{split}
    D^2\Phi(y)[u,u] =~& \eps \sum_{\xi = 1}^N A_\xi |u_\xi'|^2 + \eps \sum_{\xi = 1}^N \eps^2 B_\xi |u_\xi''|^2, \quad \text{and} \\
    D^2\Phi_\h(y)[u,u] =~& \eps \sum_{\xi = 1}^N \tilde{A}_\xi |u_\xi'|^2 + \eps \sum_{\xi \in \As} \eps^2 B_\xi |u_\xi''|^2,
  \end{split}
\end{equation}
where
\begin{align*}
  A_\xi =~& \phi''(y_\xi') + 2\phi''(y_{\xi-1}'+y_\xi')
  + 2 \phi''(y_\xi'+y_{\xi+1}'), \\
  \tilde{A}_\xi =~& \phi''(y_\xi') + \cases{ 
    2\phi''(y_{\xi-1}'+y_\xi') + 2 \phi''(y_\xi'+y_{\xi+1}'), &
    \xi \in \As', \\
    2 \phi''(y_{\xi-1}'+y_\xi') + 2 \phi''(2y_\xi'), & 
    \xi \in \Is_r', \\
    2 \phi''(y_{\xi}'+y_{\xi+1}') + 2 \phi''(2y_\xi'), &     
    \xi \in \Is_\ell', \\
    4 \phi''(2y_\xi'), & \xi \in \Cs' \setminus \Is',
  } \qquad \text{and} \\
  B_\xi =~& - \phi''(y_\xi'+y_{\xi+1}').
\end{align*}
The proof of (\ref{eq:stab:hessians}) requires purely algebraic
manipulations.

We observe that, apart from estimating the effect of replacing $A_\xi$
by $\tilde{A}_\xi$, which is the purpose of the next lemma, our main
challenge will be to understand the effect of dropping the strain
gradient term in the continuum region.

\begin{lemma}
  \label{th:error_A_lemma}
  Let $y \in \Ys$ such that $\min_{\xi \in \Z} y_\xi' > 0$, then
  \begin{displaymath}
    \|A_\xi - \tilde{A}_\xi\|_{\ell^\infty} \leq
    \eps 4 \bar{C}_3 \| y'' \|_{\ell^\infty(\Cs)},
  \end{displaymath}
  where $\bar{C}_3 = C_3(2\min_{\xi \in \Cs'} y_\xi')$.
\end{lemma}
\begin{proof}
  For $\xi \in \As'$, $A_\xi = \tilde{A}_\xi$. For $\xi \in \Cs'
  \setminus \Is'$, we have
  \begin{align*}
    A_\xi - \tilde{A}_\xi 
    =~& 2 (\phi''(y_{\xi-1}'+y_\xi') - \phi''(2y_\xi'))
    + 2 (\phi''(y_{\xi}'+y_{\xi+1}') - \phi''(2y_\xi')) \\
    =~& - \eps 2 \phi'''(\vartheta_{1,\xi}) y_{\xi-1}''
    + \eps 2 \phi'''(\vartheta_{2,\xi}) y_\xi'',
  \end{align*}
  where $\vartheta_{1,\xi} \in \conv\{ y_{\xi-1}'+y_\xi', 2 y_\xi' \}$
  and $\vartheta_{2,\xi} \in \conv\{ y_{\xi}'+y_{\xi+1}', 2 y_\xi'
  \}$. Hence, we obtain
  \begin{displaymath}
    |A_\xi - \tilde{A}_\xi| \leq 2 \eps \bar{C}_3 |y_{\xi-1}''| + 2 \eps \bar{C}_3 |y_\xi''| \leq 4 \eps \bar{C}_3 \|y''\|_{\Cs}.
  \end{displaymath}
  Performing a similar calculation for the interfaces, and sacrificing
  a factor two there, we obtain the stated result.
\end{proof}

\begin{remark}
  If $\phi \in \CC^4((0, +\infty])$ then the result of Lemma
  \ref{th:error_A_lemma} can be improved. Using similar calculations
  as in the proof of Proposition \ref{th:consistency}, one obtains
  \begin{displaymath}
    \| A_\xi - \tilde{A}_\xi \|_{\ell^\infty} \leq \max\big(
    \eps 2 \bar{C}_3 \|y''\|_{\ell^\infty(\Is)}, 
    \eps^2 \bar{C}_4 (\|y'''\|_{\ell^\infty(\Cs')} +\|y''\|_{\ell^\infty(\Cs)}^2)
    \big).
  \end{displaymath}
  Even though this result is clearly sharper than Lemma
  \ref{th:error_A_lemma}, it carries only limited practical value due
  to the fact that the maximum in $\|y''\|_{\ell^\infty(\Cs)}$ will
  typically be attained at the interface.
\end{remark}

\medskip Before formulating the first stability result, we briefly
discuss the nature of the coefficients $B_\xi$. In view of the
discussion in Section \ref{sec:bound_nnn} we have assumed throughout
that $y_\xi' \geq r_*/2$ for all $\xi \in \Z$ and consequently,
$\phi''(y_\xi' + y_{\xi+1}') \leq 0$ for all $\xi$, that is, we obtain the following result.

\begin{lemma}
  If $y_\xi' \geq r_*/2$ for all $\xi \in \Z$, then
  \begin{equation}
    \label{eq:sign_Bxi}
    B_\xi \geq 0 \qquad \forall \xi \in \Z.
  \end{equation}
\end{lemma}

\subsection{An illustrative example}
\label{sec:stab_ex}
Before we turn to the stability estimates, we briefly discuss an
example that further stresses the difference between atomistic and QC
hessians. We consider the case $y = Fx$ and $\As = \emptyset$, that
is, the QNL-QC method reduces to the local QC (or simply,
Cauchy--Born) method. In that case, we have
\begin{displaymath}
  D^2\Phi(Fx)[u,u] = A \| u' \|_{\ell^2_\eps}^2 + \eps^2 B \|u''\|_{\ell^2_\eps}
  \quad \text{and} \quad D^2 \Phi_\h(Fx)[u,u] = A \|u'\|_{\ell^2_\eps}^2,
\end{displaymath}
where $A = \phi''(F) + 4 \phi''(2F)$ and $B = - \phi''(2F)$. 

In this case, the $\Us^{1,2}$-spectrum of $D^2\Phi_\h(Fx)$ contains only
the eigenvalue $A$ and every vector $u \in \Us$ is an eigenvector. On
the other hand, the $\Us^{1,2}$-spectrum of $D^2\Phi(Fx)$ is given by
\begin{align*}
  \sigma_{\Us^{1,2}}(D^2\Phi(Fx)) = \big\{ A + B \mu_j : 
  j = 1, \dots, N-1 \big\}, 
 \quad \text{where } \mu_j = 4 \sin^2\big( \smfrac12 j \pi \eps \big);
\end{align*}
see \cite[Sec. 3.1]{Dobson:qce.stab} and \cite[Lemma
1]{Dobson:qcf.iter} for similar analyses.

Thus, we see that the eigenvalues of the low eigenmodes are well
approximated by the local QC stability constant $A$, and in
particular,
\begin{displaymath}
  c(Fx) = A + 4 B \sin^2(\pi \eps) = A + O(\eps^2) = c_\h(Fx) + O(\eps^2).
\end{displaymath}
On the other hand, the eigenvalues corresponding to the high frequency
modes are significantly larger, in other words, high frequency
perturbations are considerably more expensive in the atomistic model
than in the continuum model. This discrepancy is not too surprising
since the basic assumption of the Cauchy--Born model is the absence of
high-frequence modes in the deformation.

\subsection{An a priori stability result}
\label{sec:stab_apriori}
A common approach to establish the stability of a QC method is to
assume that next-nearest neighbour interactions are dominated by
nearest-neighbour interactions
\cite{Dobson:2008a,LinP:2006a,emingyang,Ortner:2008a}. This leads to
particularly simple results that can even be made fairly sharp as the
numerical experiments in \cite{Ortner:2008a} demonstrate. If one
wishes to include more than the simplest examples in such stability
results, then one should specify classes of atomistic solutions that
are of interest and prove, for each such class, that the QC hessian is
stable when evaluated at those deformations.  In \cite{Ortner:2008a},
two main types of stable solutions were identified for a variant of
the model problem \eqref{eq:model:crit_a}: (i) elastic deformations
and (ii) deformations with a single crack. Here we will focus only on
elastic solutions only, for which we obtain the following result.

\begin{proposition}[Stability of elastic states]
  \label{th:simple_stab_el}
  Let $y \in \Ys$, such that $y_\xi' \geq r_*/2$ for all $\xi \in \Z$,
  then
  \begin{equation}
    \label{eq:simple_stab_el}
    c(y) \geq \min_{\xi = 1, \dots, N} A_\xi,
    \quad \text{and} \quad
    c_\h(y) \geq \min_{\xi = 1, \dots, N} \tilde{A}_\xi,
  \end{equation}
  where $c(y), c_\h(y)$ are defined in \eqref{eq:approx_stab} and
  $A_\xi, \tilde{A}_\xi$ are defined in \eqref{eq:stab:hessians}.
\end{proposition}
\begin{proof}
  Both bounds in \eqref{eq:simple_stab_el} follow immediately from
  \eqref{eq:stab:hessians} and from the fact that $B_\xi \geq 0$ for
  all $\xi \in \Z$.
%
\end{proof}

\begin{remark}
  Proposition \ref{th:simple_stab_el} provides a particularly simple
  example of atomistic configurations for which the QNL-QC method is
  stable. Due to the representation of next-nearest neighbour
  interactions developed in Section \ref{sec:stab}, these results are
  in fact significantly sharper than those found in
  \cite{Dobson:2008a,emingyang,Ortner:2008a}, and will be sufficient
  whenever elastic effects dominate. Since there is no regularity
  assumption on $y$, simple singularities are included as
  well.
\end{remark}

\begin{remark}
  It must be said at this point that the stated goal, namely a proof
  of $c(y) \approx c_\h(y)$, or at least $c_\h(y) \gtrsim c(y)$, up to
  a controllable error, has not been obtained here (however, a reverse
  inequality will be established in the next section). It can be
  easily seen that such a result would be related to the regularity of
  eigenfunctions in the continuum region. Namely, let $\bar u$ be an
  eigenfunction for $D^2 \Phi_\h(y)$, associated with the eigenvalue
  $c_\h(y)$, then
  \begin{align*}
    c(y) =~& \min_{\|u'\|_{\ell^2_\eps} = 1} D^2 \Phi(y)[u,u] \\
    \leq~& D^2\Phi(y)[\bar u, \bar u] \\
    \leq~& c_\h(y) 
    + \big| D^2\Phi(y)[\bar u, \bar u] 
    - D^2\Phi_\h(y)[\bar u, \bar u]  \big|.
  \end{align*}
  Using \eqref{eq:stab:hessians} and Lemma \ref{th:error_A_lemma}, we
  obtain the estimate
  \begin{displaymath}
    \big| D^2\Phi(y)[\bar u, \bar u] 
    - D^2\Phi_\h(y)[\bar u, \bar u]  \big| \leq 4 \bar{C}_3 \eps \|y''\|_{\ell^\infty(\Cs)} + \eps^3 \sum_{\xi \in \Cs} B_\xi |\bar u_\xi''|^2.
  \end{displaymath}
  Thus, if one could establish an $O(1)$ bound on
  $\|\bar{u}''\|_{\ell^2_\eps(\Cs)}$, for example in terms of the smoothness
  of $y$ in the continuum region, then one would obtain
  \begin{displaymath}
    c_\h(y) \geq c(y) - O(\eps).
  \end{displaymath}
  However, such a bound is far from trivial to obtain, even in one
  dimension.
\end{remark}

\subsection{A posteriori stability}
\label{sec:stab_apost}
If we assume that $y = y_\h$ is a local minimizer of the QNL-QC
approximation, then we can check the condition whether
$D^2\Phi_\h(y_\h) > 0$ {\it a posteriori} by solving a (generalized)
eigenvalue problem. The question we then ought to ask is whether this
stability carries over to the full atomistic model, that is, whether
also $D^2\Phi(y_\h) > 0$. If this were not the case then it cannot be
guaranteed that a local minimizer in the QNL-QC model corresponds to a
local minimizer in the atomistic model. However, for our model problem
the answer is surprisingly simple and is given in the following
result.

\begin{theorem}[A Posteriori Stability]
  \label{th:apost_stab}
  Let $y \in \Ys$ such that $\min_{\xi \in \Z} y_\xi' > 0$, and such
  that
  \begin{displaymath}
    \phi''(y_\xi'+y_{\xi+1}') \leq 0 \qquad \forall \xi \in \Cs,
  \end{displaymath}
  then
  \begin{displaymath}
    c(y) \geq c_\h(y) - \eps 4 \bar{C}_3 \|y''\|_{\ell^\infty(\Cs)},
  \end{displaymath}
  where $\bar{C}_3 = C_3(2\min_{\xi \in \Cs'} y_\xi')$.
\end{theorem}
\begin{proof}
  From Lemma \ref{th:error_A_lemma} and the assumption that
  $\phi''(y_\xi'+y_{\xi+1}') \leq 0$, or equivalently, $B_\xi \geq 0$
  for all $\xi \in \Cs$ we have, for all $u \in \Us$,
  \begin{displaymath}
    \Phi''(y)[u,u] \geq \Phi_\h''(y)[u,u] 
    - \eps 4 \bar{C}_3 \|y''\|_{\ell^\infty(\Cs)} \|u'\|_{\ell^2_\eps}.
  \end{displaymath}
  The stated result follows by taking the infimum over $u \in \Us$
  with $\|u\|_{\Us^{1,2}} = 1$.
\end{proof}

\begin{remark}
  The analysis in \cite{Hudson:stab} shows that Theorem
  \ref{th:apost_stab} is false in higher dimenions. Namely, it is
  possible that a QC model (for example, a pure Cauchy--Born model) is
  stable at a uniform deformation $Fx$ while the atomistic model is
  unstable. Thus, at the very least, one needs to add an additional
  assumption that the deformation gradients in the continuum region
  belong to an atomistic stability regime.
\end{remark}

\subsection{A remark on defects}
Suppose now that a deformation $y \in \Ys$ has a single `crack', that
is, there exists a single index $\hat{\xi}$ such that $y_{\hat{\xi}}'
= O(N)$. Assume, for example, that $F > 0$ and that $\hat y \in \Ys$
is given by
\begin{equation}
  \label{eq:yhat_crack}
  \hat y_\xi' = \cases{1, & \text{if } \xi \neq \hat\xi \\
    F + (N-1)(F-1), & \text{otherwise}.
  }
\end{equation}
If $N$ is sufficiently large, then $A_{\hat \xi} \leq 0$ and hence
Proposition \eqref{eq:simple_stab_el} gives no useful
information. This is no accident as the following result demonstrates.

\begin{proposition}
  \label{th:simple_stab_cr}
  Suppose that there exists a cut-off radius $r_c > r_*$ such that
  $\phi = 0$ in $[r_c, +\infty)$. Let $\hat{y} \in \Ys$ be defined by
  (\ref{eq:yhat_crack}), then
  \begin{equation}
    \label{eq:simple_stab_cr}
    c(\hat y) \leq \hat{A} \eps 
    \quad \text{and} \quad
    c_\h(\hat y) \leq \hat{A} \eps,
    \quad \text{where} \quad 
    \hat{A} = \phi''(1) + 4 \phi''(2).
  \end{equation}
\end{proposition}
\begin{proof}
  There exists a unique test function $\hat{u} \in \Us$ such that
  \begin{displaymath}
    \hat{u}_\xi' = \cases{
      (N-1)^{1/2}, & \xi = \hat{\xi}, \\
      - (N-1)^{-1/2}, & \xi \neq \hat{\xi}, }
  \end{displaymath}
  and we note that $\|\hat{u}_\xi'\|_{\ell^2_\eps} = 1$.

  Next, we observe that $\hat{u}_\xi'' = 0$ for all $\xi$ except for
  $\xi = \hat{\xi},\hat{\xi}-1$. However, since $\phi = 0$ in $[r_c,
  +\infty)$, it follows that $B_{\hat{\xi}-1} = B_{\hat{\xi}} = 0$ for
  $N$ sufficiently large, and hence we can simply ignore the strain
  gradient term. Thus, testing $D^2\Phi(\hat{y})$ with $\hat{u}$, we
  obtain
  \begin{displaymath}
    D^2\Phi(\hat{y})[\hat{u},\hat{u}] =
    \eps \sum_{\xi = 1}^N A_\xi (N-1)^{-1}.
  \end{displaymath}
  
  Setting $\hat{A} = \phi''(1) + 4 \phi''(2)$, we have $A_\xi =
  \hat{A}$ for all $\xi$ except $\xi = \hat{\xi}-1, \hat{\xi},
  \hat{\xi}+1$. For $\xi = \hat\xi$, we have $A_{\hat{\xi}} \leq 0$
  provided that $N$ is sufficiently large, while, for $\xi = \hat{\xi}
  \pm 1$, it follows from the concavity of $\phi$ in $(r_*, +\infty)$
  that $A_\xi \leq \hat{A}$. Hence, we obtain
  \begin{displaymath}
    D^2\Phi(\hat{y})[\hat{u},\hat{u}] =
    \eps \sum_{\xi = 1}^N A_\xi (N-1)^{-1} \leq \eps \hat{A}. 
  \end{displaymath}

  Since $c(\hat{y}) \geq c_\h(\hat{y})$ for this example, the same
  result also holds for the QC stability constant.
\end{proof}

Since deformations with cracks or other defects can surely be stable
equilibria, this result only shows that they cannot be analyzed in the
$\Us^{1,2}$ space. Different topologies must be chosen to analyze the
stability of defects. The analysis in \cite{Ortner:2008a} suggests
that the correct norm for a 1D example with a single crack might be
\begin{displaymath}
  \| u \|_{\Us^{1,2}_{\hat{\xi}}} := \Big( \eps \sum_{\xi \neq \hat{\xi}} 
  |u_\xi'|^2 \Big)^{1/2},
\end{displaymath}
however, we will not pursue this direction further in the present
work. We note, however, that changing the norm inside the atomistic
region does not change the residual estimate. To this this, one
verifies that, in the last line of the proof of Theorem
\ref{th:consistency}, the term $\|u'\|_{\ell^{p'}_\eps}$ may be
replaced by $\|u'\|_{\ell^{p'}_\eps(\Cs')}$.

\section{Existence and Convergence}
\label{sec:existence}
In this section, we will combine the consistency and stability
analysis of Sections \ref{sec:consistency} and \ref{sec:stability} to
prove {\it a priori} and {\it a posteriori} existence results and
error estimates.

\subsection{An a priori existence result for elastic deformations}
\label{sec:exapriori}
In this section, we extend the main results in
\cite{Dobson:2008b,emingyang} to the context of large deformations and
solutions that are possibly non-smooth in the atomistic region.

We begin by assuming the existence of a strongly stable equilibrium
$y$ of the full atomistic model. If we assume that this deformation is
``sufficiently smooth'' in the continuum region then Theorem
\ref{th:consistency} shows that the truncation error is small, and
hence $y$ can be considered an approximate solution of the QNL-QC
equilibrium equations (\ref{eq:crit_qc}). If $y$ is also stable in the
QNL-QC model, then we can employ the inverse function theorem, Lemma
\ref{th:inverse_fcn_thm}, to prove the existence of an exact solution
$y_\h$ of the QNL-QC equilibrium equations in a small neighbourhood of
$y$. This idea is made rigorous in the following result, which is best
read as follows:

{\it ``If $y$ is a strongly stable local minimizer of $\Phi^\tot$,
  which is smooth in the continuum region, then there exists a
  solution $y_\h$ of the QNL-QC method, which is a good approximation
  to $y$.''}

\begin{theorem}
  \label{th:exapriori}
  Let $y \in \argmin\,\Phi^\tot$, and assume that $y_\xi' \geq r_*/2$,
  and that
  \begin{equation}
    \label{eq:exapriori:stabass}
    \min_{\xi = 1, \dots, N} A_\xi =: \underline{A} > 0.
  \end{equation}
  Then there exist constants $\delta_1, \delta_2$ that depend only on
  $\min_\xi y_\xi'$ and on $\underline{A}$, such that, if
  \begin{align}
    \label{th:exapriori_c1}
    \eps \|y''\|_{\ell^\infty(\Cs)} \leq~& \delta_1, \qquad \text{and} \\
    \label{th:exapriori_c2}
    \eps^{1/2} \|y''\|_{\ell^2_\eps(\Is)}
    + \eps^{3/2} \big(\|y'''\|_{\ell^2_\eps(\Cs' \setminus \Is')}
    + \|y''\|_{\ell^4_\eps(\Cs)}^2\big)
    \leq~& \delta_2,
  \end{align}
  then there exists a (locally unique) local minimizer $y_\h$ of
  $\Phi_\h^\tot$ in $\Ys$ such that
  \begin{displaymath}
    \| (y - y_\h)' \|_{\ell^2_\eps} \leq \smfrac{4}{\underline{A}} \big\{
    \eps \bar{C}_2 \|y''\|_{\ell^2_\eps(\Is)} + \eps^2 \bar{C}_3 \big(
    \|y'''\|_{\ell^2_\eps(\Cs' \setminus \Is')} + 
    \|y''\|_{\ell^{4}_\eps(\Cs)}^2 \big) \big\}.
  \end{displaymath}
  where $\bar{C}_i = C_i(2\min_{\xi \in \Cs'} y_\xi')$, $i = 2, 3$.
\end{theorem}
\begin{proof}
  To put the statement into the context of Lemma
  \ref{th:inverse_fcn_thm} we define $X = \Us^{1,2}$, $Y =
  \Us^{-1,2}$, $\Fs(w) = D\Phi_\h^\tot(y+w)$, and $x_0 = 0$. The set $A$
  is given by
  \begin{displaymath}
    A = \{ w \in \Us : y_\xi' + w_\xi' > 0, \xi \in \Z\}.
  \end{displaymath}
  Since $\phi \in \CC^3((0, +\infty])$, it is clear that $\Fs$ is
  continuously (Fr\'{e}chet) differentiable, with $D\Fs(0) = D^2
  \Phi_\h(y)$.

  {\it 1. Consistency: } \quad From Theorem \ref{th:consistency} we
  obtain that
  \begin{align*}
    \| \Fs(0) \|_{Y} =~& \| D\Phi_\h^\tot(y) \|_{\Us^{-1,2}}
    = \| D\Phi_\h(y) - D\Phi(y) \|_{\Us^{-1,2}} \leq \eta,
    \qquad \text{where} \\
    &~\eta = \eps \bar{C}_2 \|y''\|_{\ell^2_\eps(\Is)} + \eps^2 \bar{C}_3 
    \big( \|y'''\|_{\ell^2_\eps(\Cs' \setminus \Is')} + 
    \|y''\|_{\ell^{4}_\eps(\Cs)}^2  \big).
  \end{align*}

  {\it 2. Stability: } \quad The assumption that $y_\xi' \geq r_*/2$
  implies that $B_\xi \geq 0$ for all $\xi$. Using this fact, the
  ``elasticity assumption'' (\ref{eq:exapriori:stabass}), as well as
  Lemma \ref{th:error_A_lemma}, we obtain
  \begin{align*}
    D^2\Phi_\h(y)[u,u] \geq~& \eps\sum_{\xi = 1}^N \tilde{A}_\xi |u_\xi'|^2 \\
    \geq~& \eps\sum_{\xi = 1}^N A_\xi |u_\xi'|^2 - \eps 4 \bar{C}_3 
    \|y''\|_{\ell^\infty(\Cs)} \|u'\|_{\ell^2_\eps(\Cs)}^2 \\
    \geq~& \big(\underline{A} - \eps 4 \bar{C}_3 
    \|y''\|_{\ell^\infty(\Cs)}\big) \|u'\|_{\ell^2_\eps}^2.
  \end{align*}
  Hence, we obtain
  \begin{displaymath}
    c_\h(y) \geq \underline{A} - \eps 4 \bar{C}_3 
    \|y''\|_{\ell^\infty(\Cs)},
  \end{displaymath}
  and consequently, if $\underline{A} - \eps 4 \bar{C}_3
  \|u''\|_{\ell^\infty(\Cs)} > 0$, then
  \begin{displaymath}
    \| D\Fs(0)^{-1} \|_{L(\Us^{-1,2}, \Us^{1,2})} \leq
    \big( \underline{A} - \eps 4 \bar{C}_3 \|y''\|_{\ell^\infty(\Cs)} 
    \big)^{-1}.
  \end{displaymath}
  Setting $\delta_1 = \underline{A} / (8 \bar{C}_3)$,
  \eqref{th:exapriori_c1} becomes
  \begin{equation}
    \label{eq:exapost_bd0}
    \eps 4 \bar{C}_3 \|y''\|_{\ell^\infty(\Cs)} \leq \smfrac12 \underline{A},
  \end{equation}
  and hence we obtain that
  \begin{displaymath}
    \| D\Fs(0)^{-1} \|_{L(\Us^{-1,2}, \Us^{1,2})} \leq 
    \sigma := 2/\underline{A}.
  \end{displaymath}

  {\it 3. Lipshitz bound: } \quad Next, we compute a Lipschitz bound
  for $D\Fs = D^2\Phi_\h$. To this end, we first need to establish a
  lower bound on $y_\xi' + w_\xi'$ for all $w$ satisfying
  $\|w'\|_{\ell^2_\eps} \leq 2\eta\sigma$. Since the
  $\Us^{1,2}$-topology cannot provide this bound directly, we need to
  resort to an inverse inequality: for all $w \in \Us$ such that
  $\|w\|_{\Us^{1,2}} \leq 2 \eta \sigma$, we have
  \begin{equation}
    \label{eq:inverse1}
    \|w'\|_{\ell^\infty} \leq \eps^{-1/2} \|w'\|_{\ell^2_\eps} 
    \leq M_1' \big( \eps^{1/2} \|y''\|_{\ell^2_\eps(\Is)}
    + \eps^{3/2} (\|y'''\|_{\ell^2_\eps(\Cs' \setminus \Is')}
    + \|y''\|_{\ell^4_\eps(\Cs)}^2) \big),
  \end{equation}
  where $M_1'$ is a constant that depends only on $\underline{A}$ and on
  $\bar{C}_2, \bar{C}_3$.  Hence, if 
  \begin{equation}
    \label{eq:exapriori:bd1}
    \eps^{1/2} \|y''\|_{\ell^2_\eps(\Is)}
    + \eps^{3/2} (\|y'''\|_{\ell^2_\eps(\Cs' \setminus \Is')}
    + \|y''\|_{\ell^4_\eps(\Cs)}^2)
    \leq \smfrac{\delta}{M_1'} \min_\xi y_\xi',
  \end{equation}
  then the desired bound
  \begin{displaymath}
    y_\xi' + w_\xi' \geq \delta y_\xi'
  \end{displaymath}
  holds.

  It is now a simple exercise to show that, for any two displacements
  $w_1, w_2 \in \Us$, such that $\|w_i\|_{\Us^{1,2}} \leq 2
  \eta\sigma$, we have the Lipschitz bound
  \begin{displaymath}
    \big|D^2\Phi_\h(y+w_1)[u,v] - D^2\Phi_\h(y+w_2)[u,v]\big|
    \leq L' \| (w_1 - w_2)' \|_{\ell^\infty} \|u'\|_{\ell^2_\eps} \|v'\|_{\ell^2_\eps}
    \quad \forall u, v \in \Us,
  \end{displaymath}
  where $L'$ is a constant that depends only $C_3(\delta \min_\xi
  y_\xi')$. Applying the same inverse inequality as in
  (\ref{eq:inverse1}) we obtain that,
  \begin{displaymath}
    \| D\Fs(w_1) - D\Fs(w_2) \|_{L(\Us^{1,2}, \Us^{-1,2})}
    \leq \eps^{-1/2} L' \| w_1 - w_2 \|_{\Us^{1,2}}.
  \end{displaymath}
  We set $L = \eps^{-1/2} L'$. (At this point we could choose $\delta$
  in order to optimize the size of $L$ against the requirement
  (\ref{eq:exapriori:bd1}). Since we are only interested in a
  qualitative result, we will ignore this possibility.)

  {\it 4. Conclusion: } \quad To conclude the proof, we only need to
  check the condition under which $2 L \sigma^{2} \eta < 1$. Inserting
  the bounds for the various terms, assuming that
  (\ref{eq:exapriori:bd1}) and (\ref{eq:exapost_bd0}) hold, we obtain
  that
  \begin{displaymath}
    2 L \sigma^{2} \eta \leq M_2' \big\{ \eps^{1/2} \|y''\|_{\ell^2_\eps(\Is)}
    + \eps^{3/2} \big(\|y'''\|_{\ell^2_\eps(\Cs' \setminus \Is')}
    + \|y''\|_{\ell^4_\eps(\Cs)}^2\big) \big\}
  \end{displaymath}
  where $M_2'$ depends only on $\underline{A}$, on $\bar{C}_2,
  \bar{C}_3$, and on $L'$.

  Hence, we conclude that, if (\ref{th:exapriori_c1}) and
  \eqref{th:exapriori_c2} hold with
  \begin{displaymath}
    \delta_1 = \underline{A} / (8 \bar{C}_3) \quad \text{and} \quad
     \delta_2 = \min\big\{ \smfrac{\delta}{M_1'} \min_\xi y_\xi',
    1 / M_2' \big\},
  \end{displaymath}
  then there exists $w_\h \in \Us$ such that $F(w_\h) = 0$, or
  equivalently $D\Phi_\h^\tot(y+w_\h) = 0$, and
  \begin{displaymath}
    \| w_\h \|_{\Us^{1,2}} \leq 2 \eta \sigma \leq \smfrac{4}{\underline{A}} \eta.
  \end{displaymath}
  Taking into account the bound for $\eta$ from step 1, we obtain the
  stated error estimate.

  Local minimality of $y_\h = y + w_\h$ follows from
  \eqref{th:exapriori_c1} which implies that
  \begin{displaymath}
    c_\h(y) \geq \smfrac12 \underline{A},
  \end{displaymath}
  and the local Lipschitz bound, which guarantees
  \begin{displaymath}
    c_\h(y_\h) \geq c_\h(y) - L (2\eta\sigma) 
    > c_\h(y) - \smfrac1\sigma
    = 0.
  \end{displaymath}
  Thus $D^2\Phi_\h(y_\h)$ is positive definite and therefore $y_\h$ is
  a local minimizer.
\end{proof}

\begin{remark}
  \label{rem:apriori}
  As indicated above, Theorem \ref{th:exapriori_c1} is a variation
  or generalization of \cite[Thm. 4.1]{Dobson:2008b} and
  \cite[Thm. 5.13]{emingyang}, and hence we should briefly discuss
  the differences between these results.

  Theorem 4.1 in \cite{Dobson:2008b} treats a linearized case, but is
  otherwise fairly sharp. For example, the stability condition derived
  in \cite[Lemma 2.3]{Dobson:2008b} was later shown to be exact
  \cite{Dobson:qce.stab}. The convergence rate in terms of the atomic
  spaces $\eps$ (or $h$ in \cite{Dobson:2008b}) is optimal, however,
  the dependence of the error on the smoothness of the exact solution
  is much stronger than in Theorem \ref{th:exapriori_c1}. This is a
  consequence of deriving the trunction error estimate in an
  $\ell^p$-type norm, and comparing the QC solution with a continuum
  solution instead of the exact atomistic solution.

  Although a very different terminology is used, some aspects of the
  proof of \cite[Thm. 5.13]{emingyang} are quite similar to the
  analysis presented here. However, in \cite{emingyang} the
  consistency analysis is based on a higher order expansion of the
  atomistic solution in terms of the Cauchy--Born model. As a result,
  the analysis proves the existence of atomistic and QC solutions only
  for ``sufficiently small and sufficiently smooth'' applied forces,
  in a small neighbourhood of the reference state $Fx$. Moreover, the
  use of a different norm for measuring the truncation error
  \cite[Eq. (5.27)]{emingyang} leads to error estimates that are
  optimal in $\eps$ only in the $\Us^{1,\infty}$-norm, and it is
  unclear from the presentation how the error depends on the
  smoothness of the atomistic solution.

  The work in \cite{emingyang} is more general than the present paper
  in that it analyzes the more general {\em geometrically consistent
    QC scheme} proposed in \cite{E:2006}.
\end{remark}

\subsection{An a posteriori existence result}
\label{sec:exapost}
Theorem \ref{th:exapriori} is an {\it a priori} result, that is, it
guarantees the existence and accuracy of a QNL-QC solution under
certain assumption on a given atomistic solution. Results of this type
can guarantee that certain classes of atomistic solutions can be
reliably approximated by the QNL-QC method. In practise, however, an
{\it a posteriori} result of the same type may be even more
interesting. Given a computed QNL-QC solution $y_\h$, we may ask
whether an (exact) atomistic solution $y$ exists that $y_\h$ is an
approximation to. Questions of this kind are investigated in the
literature on {\em numerical enclosure methods} (see \cite{Plum:2001a}
for a recent review, or \cite{ortner_apostex} where this concept is
called {\em a posteriori existence}) and leads to the following
result, which is best read as follows:

{\it ``If $y_\h$ is a strongly stable QNL-QC solution, which is smooth
  in the continuum region, then there exists an exact solution $y$ of
  the full atomistic model, for which $y_\h$ is a good
  approximation.''}

\begin{theorem}
  \label{th:exapost}
  Let $y_\h \in \argmin \Phi_\h^\tot$, such that $c_\h(y_\h) > 0$ and 
  \begin{displaymath}
    \phi''\big( (y_\h)_\xi' + (y_\h)_{\xi+1}' \big) \leq 0
    \qquad \forall \xi \in \Cs.
  \end{displaymath}
  Then there exist constants $\delta_1, \delta_2$ that depend only on
  $\min_\xi (y_\h)_\xi'$ and on $c_\h(y_\h)$ such that, if
  \begin{align}
    \label{th:exapost_c1}
    \eps \|y_\h''\|_{\ell^\infty(\Cs)} \leq~& \delta_1, \qquad \text{and} \\
    \label{th:exapost_c2}
    \eps^{1/2} \|y_\h''\|_{\ell^2_\eps(\Is)}
    + \eps^{3/2} \big(\|y_\h'''\|_{\ell^2_\eps(\Cs' \setminus \Is')}
    + \|y_\h''\|_{\ell^4_\eps(\Cs)}^2\big)
    \leq~& \delta_2,
  \end{align}  
  then there exists a strongly stable local minimizer $y$ of
  $\Phi^\tot$ in $\Ys$ such that
  \begin{displaymath}
    \big\| (y_\h - y)' \big\|_{\ell^2_\eps} \leq 
    2 \frac{\eps \bar{C}_2 \|y_\h''\|_{\ell^2_\eps(\Is)} + \eps^2 \bar{C}_3
    \big( \|y_\h'''\|_{\ell^2_\eps(\Cs' \setminus \Is')} 
    + \|y_\h''\|_{\ell^{4}_\eps(\Cs)}^2 \big)}{c_\h(y_\h) - \eps 4 \bar{C}_3 \|y_h''\|_{\ell^\infty(\Cs)}}.
  \end{displaymath}
  where $\bar{C}_i = C_i(2\min_{\xi \in \Cs'} (y_\h)_\xi')$, $i = 2, 3$.
\end{theorem}
\begin{proof}
  The proof follows along very similar lines as the proof of Theorem
  \ref{th:exapriori}, and we will therefore only focus on those parts
  of the argument that change.

  To put the statement into the context of Lemma
  \ref{th:inverse_fcn_thm}, we define $X = \Us^{1,2}, Y = \Us^{-1,2}$,
  $\Fs(w) = D\Phi^\tot(y_\h + w)$, and $x_0 = 0$. The set $A$ is again
  given by
  \begin{displaymath}
    A = \big\{ w \in \Us : (y_\h)_\xi' + w_\xi' > 0, \xi \in \Z \big\}.
  \end{displaymath}
  Since $\phi \in \CC^3(0, +\infty)$, it follows that $\Fs$ is
  continuously Fr\'{e}chet differentiable, with $D\Fs(0) =
  D^2\Phi(y_\h)$.

  {\it 1. Consistency: } From Theorem \ref{th:consistency}, we obtain
  \begin{align*}
    \| \Fs(0) \|_Y =~& \|D\Phi^\tot(y_\h)\|_{\Us^{-1,2}} 
    = \| D\Phi(y_\h) - D\Phi_\h(y_\h) \|_{\Us^{-1,2}} \leq \eta, 
    \quad \text{where} \\
    &~ \eta = \eps \bar{C}_2 \|y_\h''\|_{\ell^2_\eps(\Is)} + \eps^2 \bar{C}_3 
    \big( \|y_\h'''\|_{\ell^2_\eps(\Cs' \setminus \Is')} + 
    \|y_\h''\|_{\ell^{4}_\eps(\Cs)}^2  \big).
  \end{align*}

  {\it 2. Stability: } Using Theorem \ref{th:apost_stab}, we obtain
  \begin{align*}
    c(y_\h) \geq c_\h(y_\h) - \eps 4 \bar{C}_3 \|y_\h''\|_{\ell^\infty(\Cs)}
    =: 1/\sigma,
  \end{align*}
  where $\bar{C}_3 = C_3(2\min_{\xi \in \Cs'} (y_\h)_\xi')$.  Hence,
  if we require that (\ref{th:exapost_c1}) holds with $\delta_1 =
  c_\h(y_\h) / (8 \bar{C}_3)$, then we obtain 
  \begin{displaymath}
    \| D\Fs(0)^{-1} \|_{L(\Us^{-1,2}, \Us^{1,2})} \leq 1/ c(y_\h) \leq \sigma
    \leq 2 / c_\h(y_\h).
  \end{displaymath}

  {\it 3. Lipschitz bound: } The proof of a Lipschitz bound for $D\Fs$
  is a verbatim repetition of step 3 in the proof of Theorem
  \ref{th:exapriori}. For some fixed $0 < \delta < 1$, we obtain that,
  if $\|w_i'\|_{\ell^2_\eps} \leq 2 \eta \sigma$, $i = 1,2$, then
  \begin{displaymath}
    \|D\Fs(w_1) - D\Fs(w_2)\|_{L(\Us^{1,2}, \Us^{-1,2})} \leq
    \eps^{-1/2} L' \| w_1 - w_2 \|_{\Us^{1,2}},
  \end{displaymath}
  where $L'$ depends only on $C_3(\delta \min_\xi (y_\h)_\xi')$. The
  actual Lipschitz constant is again set to $L = \eps^{-1/2} L'$.

  {\it 4. Conclusion: } The conclusion of the proof follows precisely
  step 4 in the proof of Theorem \ref{th:exapriori}. The different
  constant in the error estimate is due to the sharper definition of
  $\sigma$.
\end{proof}

\begin{remark}
  \label{rem:aposteriori}
  Theorem \ref{th:exapost} is the first rigorous {\it a posteriori}
  error estimate of this type for the QNL-QC method. However, the
  basic idea behind the result, that is, to use {\it a posteriori}
  information about the approximate solution to deduce the existence
  of an exact solution, is not new. A similar approach to the one used
  in the proof of Theorem \ref{th:exapost} can be found in the work of
  Plum \cite{Plum:2001a} on numerical enclosure methods for semilinear
  differential equations. A detailed analysis of these ideas, in the
  context of {\it a posteriori} error control for finite element
  methods was given in \cite{ortner_apostex}. The first application of
  this idea in an analysis of a QC method was presented in
  \cite[Thm. 5.1]{Ortner:2008a}. The latter result differs from
  Theorem \ref{th:exapost} not only in the type of QC method analyzed,
  but also in the choice $\Us^{1,\infty}$ for the function space
  setting. While that choice gives uniform neighbourhoods and thus
  asymptotically sharper existence conditions, a sharp stability
  result such as Theorem \ref{th:apost_stab} would be difficult to
  prove. Moreover, a generalization of \cite[Thm. 5.1]{Ortner:2008a}
  to higher dimensions seems virtually impossible.
\end{remark}

\section*{Conclusion}
In order to outline a new framework for the analysis of QC methods, we
have chosen a particularly simple setting, admitting only
second-neighbour pair interactions, and disregarding finite element
coarse graining of the continuum region. Extensions of this work to
more complex interaction models and to coarse-grained models are in
progress \cite{OrtnerWang:2009a,XHLi:3n}. The most interesting
extension, however, will be to the application-relevant two- and
three-dimensional setting. While great care was taken in this paper
that a clear path for such a generalization is available, many
interesting questions will need to be carefully considered in order to
establish the consistency and (linear) stability of the QC
method. Even the relatively basic question of consistency is still
open for a general finite range interaction model, however, the work
in \cite{E:2006} provides a promising starting point. No results on
the stability of the QC method in 2D/3D can be found in the literature
at present.

\bibliographystyle{plain}
\bibliography{qc}

\begin{thebibliography}{10}

\bibitem{arndtluskin07b}
M.~Arndt and M.~Luskin.
\newblock Error estimation and atomistic-continuum adaptivity for the
  quasicontinuum approximation of a {F}renkel-{K}ontorova model.
\newblock {\em SIAM J. Multiscale Modeling \& Simulation}, 7:147--170, 2008.

\bibitem{arndtluskin07c}
M.~Arndt and M.~Luskin.
\newblock Goal-oriented adaptive mesh refinement for the quasicontinuum
  approximation of a {Frenkel-Kontorova} model.
\newblock {\em Computer Methods in Applied Mechanics and Engineering},
  197:4298--4306, 2008.

\bibitem{legoll}
X.~Blanc, C.~Le~Bris, and F.~Legoll.
\newblock Analysis of a prototypical multiscale method coupling atomistic and
  continuum mechanics.
\newblock {\em M2AN Math. Model. Numer. Anal.}, 39(4):797--826, 2005.

\bibitem{Dobson:2008a}
M.~Dobson and M.~Luskin.
\newblock Analysis of a force-based quasicontinuum approximation.
\newblock {\em M2AN Math. Model. Numer. Anal.}, 42(1):113--139, 2008.

\bibitem{Dobson:2008b}
M.~Dobson and M.~Luskin.
\newblock An optimal order error analysis of the one-dimensional quasicontinuum
  approximation.
\newblock {\em SIAM Journal on Numerical Analysis}, 47(4):2455--2475, 2009.

\bibitem{Dobson:qcf.iter}
M.~Dobson, M.~Luskin, and C.~Ortner.
\newblock Iterative methods for the force-based quasicontinuum approximation,
  2009.
\newblock arXiv.org:0910.2013.

\bibitem{Dobson:qcf.stab}
M.~Dobson, M.~Luskin, and C.~Ortner.
\newblock Sharp stability estimates for the force-based quasicontinuum method,
  2009.
\newblock arXiv.org:0907.3861.

\bibitem{Dobson:arXiv0903.0610}
M.~Dobson, M.~Luskin, and C.~Ortner.
\newblock Stability, instability, and error of the force-based quasicontinuum
  approximation.
\newblock {\em arXiv/0903.0610}, 2009.
\newblock to appear in Arch. Rat. Mech. Anal.

\bibitem{Dobson:qce.stab}
Matthew Dobson, Mitchell Luskin, and Christoph Ortner.
\newblock Accuracy of quasicontinuum approximations near instabilities, 2009.
\newblock arXiv.org:0905.2914.

\bibitem{E:2006}
W.~E, J.~Lu, and J.Z. Yang.
\newblock {Uniform accuracy of the quasicontinuum method}.
\newblock {\em {Phys. Rev. B}}, 74(21):214115, 2004.

\bibitem{E:2007a}
W.~E and P.~Ming.
\newblock Cauchy-{B}orn rule and the stability of crystalline solids: static
  problems.
\newblock {\em Arch. Ration. Mech. Anal.}, 183(2):241--297, 2007.

\bibitem{Gunzburger:2008a}
M.~Gunzburger and Y.~Zhang.
\newblock A quadrature-rule type approximation for the quasicontinuum method.
\newblock to appear in SIAM Multiscale Model.

\bibitem{Hudson:stab}
T.~Hudson and C.~Ortner.
\newblock On the stability of atomistic energies and their cauchy--born
  approximations.
\newblock in preparation.

\bibitem{XHLi:3n}
X.H. Li and M.~Luskin.
\newblock Analysis of quasicontinuum methods with general interactions.
\newblock manuscript (2009).

\bibitem{LinP:2006a}
P.~Lin.
\newblock Convergence analysis of a quasi-continuum approximation for a
  two-dimensional material without defects.
\newblock {\em SIAM J. Numer. Anal.}, 45(1):313--332 (electronic), 2007.

\bibitem{Luskin:clusterqc}
M.~Luskin and C.~Ortner.
\newblock An analysis of node-based cluster summation rules in the
  quasicontinuum method.
\newblock {\em SIAM Journal on Numerical Analysis}, 47(4):3070--3086, 2009.

\bibitem{Miller:2008}
R.~Miller and E.~Tadmor.
\newblock Benchmarking multiscale methods.
\newblock {\em Modeling and Simulation in Materials Science and Engineering},
  submitted.

\bibitem{Miller:2003a}
R.E. Miller and E.B. Tadmor.
\newblock {The Quasicontinuum Method: Overview, Applications and Current
  Directions}.
\newblock {\em Journal of Computer-Aided Materials Design}, 9:203--239, 2003.

\bibitem{emingyang}
P.~Ming and J.~Z. Yang.
\newblock Analysis of a one-dimensional nonlocal quasi-continuum method.
\newblock {\em Multiscale Modeling \& Simulation}, 7(4):1838--1875, 2009.

\bibitem{Ortiz:1995a}
M.~Ortiz, R.~Phillips, and E.~B. Tadmor.
\newblock {Quasicontinuum Analysis of Defects in Solids}.
\newblock {\em {Philosophical Magazine A}}, 73(6):1529--1563, 1996.

\bibitem{ortner_apostex}
C.~Ortner.
\newblock A posteriori existence in numerical computations.
\newblock {\em SIAM Journal on Numerical Analysis}, 47(4):2550--2577, 2009.

\bibitem{Ortner:2008a}
C.~Ortner and E.~S{\"u}li.
\newblock Analysis of a quasicontinuum method in one dimension.
\newblock {\em M2AN Math. Model. Numer. Anal.}, 42(1):57--91, 2008.

\bibitem{OrtnerWang:2009a}
C.~Ortner and H.~Wang.
\newblock Coarse graining in the quasicontinuum method.
\newblock In preparation.

\bibitem{Plum:2001a}
M.~Plum.
\newblock Computer-assisted enclosure methods for elliptic differential
  equations.
\newblock {\em Linear Algebra Appl.}, 324(1-3):147--187, 2001.
\newblock Special issue on linear algebra in self-validating methods.

\bibitem{prud06}
S.~Prudhomme, P.~T. Bauman, and J.~T. Oden.
\newblock Error control for molecular statics problems.
\newblock {\em International Journal for Multiscale Computational Engineering},
  4(5-6):647--662, 2006.

\bibitem{Shimokawa:2004}
T.~Shimokawa, J.J. Mortensen, J.~Schiotz, and K.W. Jacobsen.
\newblock {Matching conditions in the quasicontinuum method: Removal of the
  error introduced at the interface between the coarse-grained and fully
  atomistic region}.
\newblock {\em {Phys. Rev. B}}, 69(21):214104, 2004.

\end{thebibliography}

\end{document}